\documentclass[12pt]{amsart}

\usepackage{latexsym, amsmath, amscd, amssymb, amsthm} 
\usepackage[T2A]{fontenc}
 \usepackage[cp1251]{inputenc}
\usepackage[english]{babel}

\newcommand{\slfrac}[2]{\left.#1\middle/#2\right.}
 \newtheorem{definition}{Definition}
 \newtheorem{theorem}{Theorem}
\newtheorem{corollary}{Corollary}
\begin{document}
\author{Bedratyuk L., Luno N.}
\title{Some properties of generalized  hypergeometric Appell polynomials}



\begin{abstract}
{In this paper, we present a new  real-valued Appell-type polynomial family~$A_n^{(k)}(m,x), $ $n,  m \in {\mathbb{N}}_0,$ 
 $k \in {\mathbb{N}},$ every member of which  is expressed by mean of the generalized hypergeometric function 
$
{}_{p} F_q
\begin{bmatrix}
\begin{matrix}
a_1, a_2, \ldots, a_p \:\\
b_1, b_2, \ldots, b_q
\end{matrix} \:
 \vrule \:z
\end{bmatrix}=
\sum_{k=0}^{\infty} \frac{a_1^{(k)} a_2^{(k)} \ldots a_p^{(k)}}{b_1^{(k)} b_2^{(k)} \ldots b_q^{(k)}} \frac{z^k}{k!},
$
 where $x^{(n)}$ denotes the Pochhammer symbol (rising factorial) defined by
$x^{(n)}=x(x+1)(x+2)\cdots (x+n-1)$ for $n\geq 1$ and $x^{(0)}=1,$
  as follows
$$
A_n^{(k)}(m,x)=
x^n{}_{k+p} F_q
\begin{bmatrix}
\begin{matrix}
{a_1}, {a_2}, {\ldots}, {a_p}, {\displaystyle -\frac{n}{k}}, {\displaystyle -\frac{n-1}{k}}, {\ldots}, {\displaystyle-\frac{n-k+1}{k}}\:\\
{b_1}, {b_2}, {\ldots}, {b_q}
\end{matrix} \:
 \vrule \: \displaystyle \frac{m}{x^k}
\end{bmatrix}
$$
and is the Appell-type polynomial family simultaneously.

The generating exponential function of this type of polynomials is firstly discovered and the proof that they are of Appell-type ones is given. We present the differential operator formal power series representation  as well as an explicit formula over the standard basis, and establish a new identity for the generalized hypergeometric function. Besides, we derive  the addition, the multiplication and some other formulas for this polynomial family.
}

\end{abstract}
\maketitle

\section{Introduction}  

In \cite{Appell}, P. Appell presented polynomial sequence  $\{A_n(x) \},  n=0,1,2,\ldots,$ such that $deg A_n(x)=n$ and satisfying the identity
\begin{equation*}\label{app22}
A_n^{\prime}(x)=nA_{n-1}(x),
\end{equation*}
where $A_0(x)\neq 0,$ which is called  the Appell polynomials sequence.

An arbitrary Appell  polynomial sequence possesses an exponential generating function

\begin{equation*}\label{app1}
A(t)e^{xt}=
\sum_{n=0}^{\infty}
A_n(x)\frac{t^n}{n!},
\end{equation*}
here $A(t)$ is a formal power series 
\begin{equation}\label{tree}
A(t)=
a_0+a_1t+a_2\frac{t^2}{2!}+\cdots+a_n\frac{t^n}{n!}+\cdots \quad, a_0\neq 0. 
\end{equation}
The Appell-type polynomials $A_n(x)$ are expressed  in the terms of $\{ a_n \}$ as follows
$$
A_n(x)=\sum_{i=0}^n \binom{n}{i} a_{n-i} x^i.
$$

The simplest example of Appell-type polynomials is the monomial sequence $\{x^n\}, n=0, 1, \cdots,$ other examples are the Bernoulli, the Euler polynomials and the Hermite polynomials. For more examples one can consult \cite{Abra, Haze}.

The Appell-type polynomials perform a large variety of features and are widely spread at the different areas of mathematics, namely, at special functions, general algebra, combinatorics and number theory. 
Recently, the Appell-type polynomials are of big interest. The modern researches  give the alternative definitions of Appell-type polynomials and apply new approaches based, for instance, on the determinant method or in Pascal matrix method (see, e.g., \cite{Aldo},
 \cite{Micek}). Consequently, many new properties of those  polynomials are described and a great deal of identities involving Appell-type polynomials are obtained (see \cite{Aceto, Costa, Cheih}).

Let us recall that the generalized hypergeometric function is defined as follows

\begin{equation}\label{defdef}
{}_{p} F_q
\begin{bmatrix}
\begin{matrix}
a_1, a_2, \ldots, a_p \:\\
b_1, b_2, \ldots, b_q
\end{matrix} \:
 \vrule \: z
\end{bmatrix}=
\sum_{k=0}^{\infty} \frac{a_1^{(k)} a_2^{(k)} \ldots a_p^{(k)}}{b_1^{(k)} b_2^{(k)} \ldots b_q^{(k)}} \frac{z^k}{k!},
\end{equation}
 where 
 $a_1, a_2, \ldots, a_p$, $b_1, b_2, \ldots, b_p$  are  complex parameters and none of  $b_i$ equals to a non-positive integer or zero, 
 $x^{(n)}$ denotes the Pochhammer symbol (or rising factorial) defined by
$x^{(n)}=x(x+1)(x+2)\cdots (x+n-1)$ for $n\geq 1$ and $x^{(0)}=1.$
Further on, we denote the generalized hypergeometric function by $ {}_p F_{q}$ for brevity.

 We note that the Gauss hypergeometric function ${}_2 F_{1}$ and the Kummer hypergeometric function  ${}_1 F_{1}$ are the partial cases of (\ref{defdef}).
 
 Apart from the Appell-type polynomials, there exist some polynomial families 
which admit representation via the partial cases of the generalized hypergeometric function, i.e.,
the Jacobi polynomials (\cite{Abra})

$$
P_{n}^{\left( \alpha, \beta \right)}(z)=
\frac{{\left(  \alpha +1 \right)}^{(n)}}{n!}
{}_{2} F_1
\begin{bmatrix}
\begin{matrix}
-n, n{+}\alpha {+}\beta {+}1  \:\\
\alpha {+}1
\end{matrix} \:
 \vrule \: \displaystyle \frac{1-z}{2}
\end{bmatrix}
$$

At the same time, there exists a number of the Appell-type polynomial families which also admit the representation via  partial cases of the Gauss hypergeometric function. It is known (\cite{Abra}) that 
the Laguerre polynomials $L_n(x)$ are presented as follows
$$L_n(x)=
{}_{1} F_1
\begin{bmatrix}
\begin{matrix}
-n \:\\
1
\end{matrix} \:
 \vrule \: x
\end{bmatrix}.
$$
Remarkably, the Hermite polynomials $H_n(x)$ are simply expressed in the terms of those functions (\cite{Dom})

$$
H_n(x)=x^n
{}_{2} F_0
\begin{bmatrix}
\begin{matrix}
 {\displaystyle -\frac{n}{2}}, {\displaystyle -\frac{n-1}{2}} \:\\
-
\end{matrix} \:
 \vrule \: \displaystyle -\frac{2}{x^2}
\end{bmatrix},\quad
G(x,t)=e^{ xt-\frac{1}{2}t^2 }.
$$ 

The natural way of generalisation of the Hermitte polinomials is to expand the array of ratios for another denominators, it was made  in \cite{Gould}, the authors  obtained the Gould-Hopper polynomials $g_n^m(x,h),$ with $\displaystyle
G(x,t)=e^{xt+ht^m},
$
which could be also expressed in the terms of the generalized hypergeometric function as follows
 
$$
g_n^m(x,h)=
x^n{}_{m} F_0
\begin{bmatrix}
\begin{matrix}
 {\displaystyle -\frac{n}{m}}, {\displaystyle -\frac{n-1}{m}}, {\ldots}, {\displaystyle-\frac{n-m+1}{m}}\:\\
-
\end{matrix} \:
 \vrule \: \displaystyle \frac{{(-1)}^m h m^m}{x^m}
\end{bmatrix}
$$

The aim of this paper is to find a polynomial family, which would be the Appell-type one and admit the generalized hypergeometric function representation simultaneously.
Still, there exist the polynomial families which have the needed representation, e.g., \textit{the generalized hypergeometric polynomials} $f_n\left(a_i;b_j;x \right),$ studied at \cite{Fase}, such that

$$
f_n\left(a_i; b_j; x \right)=
{}_{p+2} F_{q+2}
\begin{bmatrix}
\begin{matrix}
-n, n+1, {a_1}, {a_2}, {\ldots}, {a_p}\:\\
1, \displaystyle \frac{1}{2}, b_{1}, {b_2}, {\ldots}, {b_q}
\end{matrix} \:
 \vrule \: x^n
\end{bmatrix}, \quad  n \in {\mathbb{N}}_0,
$$
and \textit{the incomplete hypergeometric polynomials} associated  with generalized incomplete hypergeometric function, studied at \cite{RSriva},
 but they both  are not the Appel-type polynomials.

The difference between all mentioned classes of polynomials, depending, if they are of Appell-type  or not and if they possess  the generalized hypergeometric function representation or do not, has motivated the title of the paper.

Therefore, let us give the following

\begin{definition} \label{def1}
 Let $\Delta(k,-n)$ denote the array of $k$ ratios 
$$-\frac{n}{k}, -\frac{n-1}{k}, \ldots, -\frac{n-k+1}{k}, \quad n \in {\mathbb{N}}_0,k \in \mathbb{N}.$$
Then we call the polynomial family
 
\begin{equation}\label{three}
A_n^{(k)}(m,x)=
x^n{}_{k+p} F_q
\begin{bmatrix}
\begin{matrix}
{a_1}, {a_2}, {\ldots}, {a_p}, \Delta(k,-n)\:\\
{b_1}, {b_2}, {\ldots}, {b_q}
\end{matrix} \:
 \vrule \: \displaystyle \frac{m}{x^k}
\end{bmatrix}, \quad n, m \in {\mathbb{N}}_0, k \in {\mathbb{N}}
\end{equation}
 where 
 \begin{equation}\label{four}
 {}_{k+p} F_q=
\sum_{i=0}^{\left[ \slfrac{n}{k} \right]}
\frac{\prod_{r=1}^p {\left({a}_r\right)}^{(i)}}
{\prod_{s=1}^q {\left({b}_s\right)}^{(i)}}
{\prod_{j=1}^k {\left( -\frac{n-j+1}{k} \right)}^{(i)}}
\frac{m^i}{i!x^{ki}},
\end{equation}
the generalized hypergeometric Appell polynomials.
 \end{definition}

   We note that if $p=0, q=0,$  $k:=m,$  $m:=(-1)^k h{k^k}$
     the generalized hypergeometric Appell polynomials $A_n^{(k)}(m,x)$ become the Gould-Hopper polynomials $g_n^m(x,h)$  and if $p=0, q=0, m=-2, k=2$ 
 they become the Hermite polynomials $H_n(x)$ mentioned above.

The main result of this article is the following basic statement.

\begin{theorem}\label{t1}
The generalized hypergeometric Appell polynomials $A_n^{(k)}(m,x)$ defined by definition \ref{def1} are the Appell type ones.  
\end{theorem}

\section{Basic  definitions and notation} 

In addition to the rising factorial we use the falling factorial
$
(x)_n=x(x-1)(x-2)\cdots (x-n+1)$ for $n>0$ and $(x)_0=1.$
In these notation, the following relations holds (see\cite{Abra})
\begin{equation}\label{fall}
(x)_n={(-1)}^n{(-x)}^{(n)},
\end{equation}
and the  Gauss product of indexes formula (see \cite{Karls}) will be written as follows
\begin{equation}\label{gauss}
{\left(-\lambda\right)}^{(mn)}=
m^{mn}
\prod_{j=1}^m {\left( -\frac{\lambda-j+1}{m} \right)}^{(n)}, n \in {\mathbb{N}}_0.
\end{equation}
We note that in the case when either $a$ or $b$ is a non-positive integer, the generalized hypergeometric function reduces to a polynomial:

$$
{}_{p} F_q
\begin{bmatrix}
\begin{matrix}
{-m}, {a_2}, {\ldots}, {a_p}\:\\
{b_1}, {b_2}, {\ldots}, {b_q}
\end{matrix} \:
 \vrule \:  z
\end{bmatrix}=
\sum_{n=0}^{\infty}
{\left(-1\right)}^n
{m \choose n}
\frac{\prod_{j=2}^p {{a}_j}^{(n)}
}{\prod_{s=1}^q {{b}_s}^{(n)}}{z^n}.
$$

As far as we deal with the differentiation, the differentiation formula with respect to $z$ would be useful:
$$
\frac{d}{dx}
{}_{p} F_q
\begin{bmatrix}
\begin{matrix}
{a_1}, {a_2}, {\ldots}, {a_p} \:\\
{b_1}, {b_2}, {\ldots}, {b_q}
\end{matrix} \:
 \vrule \: z
\end{bmatrix}=
\frac{\prod_{j=1}^p {a}_j}{\prod_{s=1}^q {b}_s}
{}_{p} F_q
\begin{bmatrix}
\begin{matrix}
{a_1+1}, {a_2+1}, {\ldots}, {a_p+1} \:\\
{b_1+1}, {b_2+1}, {\ldots}, {b_q+1}
\end{matrix} \:
 \vrule \: z
\end{bmatrix} \text{\cite{SLA}}.
$$

\section{Basic properties of the generalized hypergeometric Appell polynomials }
\subsection{Being of Appell type.}
\textit{Proof of theorem~\ref{t1}.}
To prove the generalized hypergeometric Appell polynomials $A_n^{(k)}(m,x)$ are the Appell-type polynomials, it is sufficient to show that there  exists a formal power series $A(t)$ such that the following relation holds
\begin{gather*}
A(t)e^{xt}=
\sum_{n=0}^{\infty}
A_n^{(k)}(m,x)\frac{t^n}{n!}.
\end{gather*}

We set
${\left(\gamma \right)}^i=\slfrac{\left(\prod_{r=1}^p {\left({a}_r\right)}^{(i)}\right)}{\left(\prod_{s=1}^q {\left({b}_s\right)}^{(i)}\right)}.
$ 
Then from definition (\ref{defdef}) and relations (\ref{fall}) and (\ref{gauss}) it follows that
\begin{gather*}
A_n^{(k)}(m,x)=
x^n{}_{p+k} F_q
\begin{bmatrix}
\begin{matrix}
{a_1}, {a_2}, {\ldots}, {a_p}, \Delta(k,-n)\:\\
{b_1}, {b_2}, {\ldots}, {b_q}
\end{matrix} \:
 \vrule \: \displaystyle \frac{m}{x^k}
\end{bmatrix}
=
x^n
\sum_{i=0}^{\left[ \slfrac{n}{k} \right]}
\frac{{\left(\gamma \right)}^i {(-1)}^{ki} {\left( n \right)}_{ki}} {k^{ki}}\frac{m^i}{i!x^{ki}}.
\end{gather*}

We choose
\begin{equation}\label{genf}
A(t)=
{}_{p} F_q
\begin{bmatrix}
\begin{matrix}
a_1, a_2, \ldots, a_p \:\\
b_1, b_2, \ldots, b_q
\end{matrix} \:
 \vrule \:  \displaystyle {(-1)}^k m\frac{t^k}{k^k}
\end{bmatrix}.
\end{equation}

Using the expansion of $e^{xt}$ into the power series and changing the product of the series by the double series, we transform the generating function as follows
\begin{align*}
A(t)e^{xt}=
\left(\sum_{n=0}^{\infty}
{\left(\gamma \right)}^n
\frac{{\left( {(-1)}^k m\frac{ t^k}{k^{k}}\right)}^n}{n!}\right)
\left(\sum_{s=0}^{\infty}
\frac{{(xt)}^s}{s!}\right)=
\sum_{n=0}^{\infty}
\left(\sum_{s=0}^{\infty}
{\left(\gamma \right)}^n {(-1)}^{kn}
\frac{m^n x^s}{k^{kn}}
\frac{ t^{s+kn}}{s!n!}\right).
\end{align*}

Using the infinite sums interchange formula (\cite{Arf})
$$
\sum_{n=0}^{\infty}
\sum_{m=0}^{\infty}a_{n, m}=
\sum_{p=0}^{\infty}
\sum_{q=0}^{p}a_{p-q,q}
$$
and taking into account the multiplicity of $i,$ we have

$$
\sum_{n=0}^{\infty}
\sum_{s=0}^{\infty}a_{s, n}=
\sum_{n=0}^{\infty}
\sum_{i=0}^{\left[ \slfrac{n}{k} \right]}
a_{n-ki,i},
$$
then
\begin{align*}
\sum_{n=0}^{\infty}
\left(\sum_{s=0}^{\infty}
{\left(\gamma \right)}^n
\frac{m^n {(-1)}^{kn} x^s}{k^{kn}}
\frac{ t^{s+kn}}{s!n!}\right)=
\sum_{n=0}^{\infty}
\left(\sum_{i=0}^{\left[ \slfrac{n}{k} \right]}
\frac{m^i {(-1)}^{ki}{\left(\gamma \right)}^i }{k^{ki}}
\frac{x^{n-ki} t^{n}}{(n-ki)!i!}\right)\notag\\
=
\sum_{n=0}^{\infty}
x^n
\left(\sum_{i=0}^{\left[ \slfrac{n}{k} \right]}
\frac{n!}{(n-ki)!i!}
\frac{m^i{(-1)}^{ki}{\left(\gamma \right)}^i }{k^{ki}}
\frac{1}{x^{ki}}\right)
\frac{t^n}{n!}=
\sum_{n=0}^{\infty}
x^n
\left(\sum_{i=0}^{\left[ \slfrac{n}{k} \right]}
{\left(\gamma \right)}^i
\frac{m^i {(-1)}^{ki}(n)_{ki} }{k^{ki}}
\frac{\frac{1}{x^{ki}}}{i!}\right)
\frac{t^n}{n!}.
\end{align*}

The inner sum is precisely equal to the generalized hypergeometric function in the form of (\ref{three}) and, therefore, the relation (\ref{four}) holds.
This means that the generating function admit the needed representation (\ref{three}).

It should be noted that there is another way to prove theorem~\ref{t1}, which is to   replace
$xt$ by $t$  and $\slfrac{m}{x^k}$ by $x$
 in problem 26, p.173
\cite{Sriva}.

As a consequence of theorem\ref{t1}, we derive a new identity for the generalized hypergeometric function.
\begin{corollary}
  The following identity holds
\begin{gather*}n x^{n-1} {}_{p+k} F_q
\begin{bmatrix}
\begin{matrix}
{a_1}, {a_2}, {\ldots}, {a_p}, \Delta(k,-n+1)\:\\
{b_1}, {b_2}, {\ldots}, {b_q}
\end{matrix} \:
 \vrule \: \displaystyle \frac{m}{x^k}
\end{bmatrix}
=n  x^{n-1} {}_{p+k} F_q
\begin{bmatrix}
\begin{matrix}
{a_1}, {a_2}, {\ldots}, {a_p}, \Delta(k,-n)\:\\
{b_1}, {b_2}, {\ldots}, {b_q}
\end{matrix} \:
 \vrule \: \displaystyle \frac{m}{x^k}
\end{bmatrix}
\\-km {\gamma}_1 {\Delta}_1(k,-n)x^{n-k-1} {}_{p+k} F_q
\begin{bmatrix}
\begin{matrix}
{a_1+1}, {a_2+1}, {\ldots}, {a_p+1}, \Delta(k,-n+k)\:\\
{b_1+1}, {b_2+1}, {\ldots}, {b_q+1}
\end{matrix} \:
 \vrule \: \displaystyle \frac{m}{x^k}
\end{bmatrix},
\end{gather*}
where ${\Delta}_1(k,-n)$ denotes the product
$$
{{\left(-\frac{n}{k}\right)}} {{\cdot \left(-\frac{n-1}{k}\right)}} {{\ldots  \left( -\frac{n-k+1}{k}\right)}}.
$$
\end{corollary}

\begin{proof}
The generalized hypergeometric Appell polynomials are the Appell-type ones, hence,   the identity
 \begin{equation*}
\frac{d}{dx} \left\lbrace
A_n^{(k)}(m,x)
\right\rbrace
=n{A_{n-1}^{(k)}(m,x)}
\end{equation*}
fulfils.

Representing the polynomials ${A_{n-1}^{(k)}(m,x)}$ in the terms of the generalized hypergeometric function according to the definition~\ref{def1},  we  immediately obtain the left side of the corollary equality. 

To obtain its right side we differentiate the hypergeometric  representation of the polynomials  ${A_n^{(k)}(m,x)}$  under the Leibnitz rule:

\begin{gather*}
\frac{d}{dx}
\left\lbrace
 x^n {}_{p+k} F_q
\begin{bmatrix}
\begin{matrix}
{a_1}, {a_2}, {\ldots}, {a_p}, \Delta(k,-n)\:\\
{b_1}, {b_2}, {\ldots}, {b_q}
\end{matrix} \:
 \vrule \: \displaystyle \frac{m}{x^k}
\end{bmatrix} \right\rbrace\\
 =n x^{n-1}
 {}_{p+k} F_q
\begin{bmatrix}
\begin{matrix}
{a_1}, {a_2}, {\ldots}, {a_p}, \Delta(k,-n)\:\\
{b_1}, {b_2}, {\ldots}, {b_q}
\end{matrix} \:
 \vrule \: \displaystyle \frac{m}{x^k}
\end{bmatrix}\\
+
 x^n \frac{d}{dx}
 \left\lbrace {}_{p+k} F_q
\begin{bmatrix}
\begin{matrix}
{a_1}, {a_2}, {\ldots}, {a_p}, \Delta(k,-n)\:\\
{b_1}, {b_2}, {\ldots}, {b_q}
\end{matrix} \:
 \vrule \: \displaystyle \frac{m}{x^k}
\end{bmatrix} \right\rbrace.
\end{gather*}
Performing the derivative of the hypergeometric function, we obtain
\begin{gather*}
x^n\frac{d}{dx}
 \left\lbrace {}_{p+k} F_q
\begin{bmatrix}
\begin{matrix}
{a_1}, {a_2}, {\ldots}, {a_p}, \Delta(k,-n)\:\\
{b_1}, {b_2}, {\ldots}, {b_q}
\end{matrix} \:
 \vrule \: \displaystyle \frac{m}{x^k}
\end{bmatrix} \right\rbrace
 =
 x^n\left(
 \frac{{(-1)}^k(n)_k}{k^k}
  \frac{a_1\cdots a_p}{b_1\cdots b_q}
  \frac{m(-k)}{1!x^{k+1}}
 \right.\\
\left.+\frac{{(-1)}^{2k}(n)_{2k}}{k^{2k}} \frac{a_1(a_1+1)\cdots a_p(a_p+1)}{b_1(b_1+1)\cdots b_q(b_q+1)}\frac{m^2(-2k)}{2!x^{2k+1}}\right.\\
\left.+\frac{{(-1)}^{3k}(n)_{3k}}{k^{3k}} \frac{a_1(a_1+1)(a_1+2)\cdots a_p(a_p+1)(a_p+2)}{b_1(b_1+1)(b_1+2)\cdots b_q(b_q+1)(b_q+2)}\frac{m^3(-3k)}{3!x^{3k+1}}+\cdots
 \right)
\end{gather*}
\begin{gather*}
=x^{n-k-1}
mk \frac{{(-1)}^{k+1}(n)_k}{k^k}
  \frac{a_1\cdots a_p}{b_1\cdots b_q}
 \left(
 1+\frac{{(-1)}^{k}(n-k)_{k}}{k^{k}} \frac{(a_1+1)\cdots (a_p+1)}{(b_1+1)\cdots (b_q+1)}\frac{m\cdot 2}{2!x^{k}} \right.\\
\left. +\frac{{(-1)}^{2k}(n-k)_{2k}}{k^{2k}} \frac{(a_1+1)(a_1+2)\cdots (a_p+1)(a_p+2)}{(b_1+1)(b_1+2)\cdots (b_q+1)(b_q+2)}\frac{m^2 \cdot 3}{3!x^{2k}} +\cdots \right)\\=-km {\gamma}_1 {\Delta}_1(k,-n) x^{n-k-1}{}_{p+k} F_q
\begin{bmatrix}
\begin{matrix}
{a_1+1}, {a_2+1}, {\ldots}, {a_p+1}, \Delta(k,-n+k)\:\\
{b_1+1}, {b_2+1}, {\ldots}, {b_q+1}
\end{matrix} \:
 \vrule \: \displaystyle \frac{m}{x^k}
\end{bmatrix},
 \end{gather*}
that ends the proof.

\end{proof}
Since an arbitrary polynomial on one variable  $P_n(x) \in \mathbb{C}[x] $ always  permits the  formal series representation

\begin{equation*}
P_n(x)=\sum_{i=0}^n{\alpha}_ix^i,
\end{equation*}
then we are interested in finding those representation for the generalized hypergeometric Appell polynomials.

\begin{corollary}\label{lm1}
The generalized hypergeometric Appell polynomials $A_n^{(k)}(m,x)$ possess 

(i) the
standard basis $\{x^i\}_{i=0}^{n}$  representation  

\begin{equation}\label{sumcoef}
A_n^{(k)}(m,x)=
\sum_{i=0}^{\left[ \slfrac{n}{k} \right]}
\frac{n!{(-1)}^{ki}{\left(\gamma \right)}^i m^i}{i!k^{ki}(n-ki)!}{x}^{n-ki},
\end{equation} 

(ii) the differential operator  formal power series
representation 
\begin{equation}\label{scoef}
A_{n}^{(k)}(m,x)=\left(\sum_{i=0}^{\left[ \slfrac{n}{k} \right]} 
\frac{{(-1)}^{ki} {\left(\gamma \right)}^i m^i }{i!k^{ki}} D^{ki} \right)
 x^{n}.
\end{equation}
\end{corollary}

\begin{proof}
(i) We use an approach from \cite{Cheih} which is based on the idea of the connection problem.  

Given the two polynomial families of Appell type $\{P_n(x)\}$ and $\{Q_n(x)\}$ with generating functions  $A_1(t)$ and  $A_2(t)$ respectively, 
the solution of its connection problem
could be written as follows:
\begin{equation*}
Q_n(x)=
\sum_{m=0}^{n}
\frac{n!}{m!}
{\alpha}_{n-m}
P_m(x),
\end{equation*}
where
$$\frac{A_2(t)}{A_1(t)}=
\sum_{k=0}^{\infty}
{\alpha}_{k}t^k.
$$

We are searching for the unknown coefficients ${\alpha}_{k}$ to decompose the polynomials 
$$Q_n(x)=x^n, A_2(t)=1$$
 upon the polynomials $A_n^{(k)}(m,x)$ defined by (\ref{three}) with generating function $A_1(t)$ defined by (\ref{genf}). 
Deriving the ratio of generating functions $A_2(t)$ and $ A_1(t)$  we have

\begin{align*}
\frac{A_2(t)}{A_1(t)}=
\sum_{r=0}^{\infty}
\frac{{(-1)}^{kr} m^r}{k^{kr}r!}
{\left(\gamma \right)}^r
 t^{kr}
 =
 \sum_{r=0}^{\infty}
 {\alpha}_{{rk}}
 t^{rk},
\end{align*}
 and, constructing the corresponding coefficients ${\alpha}_{n-m},$ we obtain the needed representation.
 
 (ii) 
 An arbitrary Appell-type polynomial $
P_n(x)$ could be also written in the symmetric form 
\begin{equation*}\label{symm}
P_n(x)=\sum_{i=0}^n {n \choose i}c_i x^{n-i}.
\end{equation*}

According to \cite{Haze}, the latter expression is equivalent to the following differential operator representation
$$
P_n(x)=\left( \sum_{i=0}^n 
\frac{c_i}{i!} D^i \right)
 x^{n},
$$
where $D: = \slfrac{d}{dx}$ is an ordinary differentiation with respect to x,
consequently, 
$$
A_n^{(k)}(m,x)=\sum_{i=0}^{\left[ \slfrac{n}{k} \right]} {n \choose ki}c_i x^{n-ki}=\sum_{i=0}^{\left[ \slfrac{n}{k} \right]} {n \choose ki}
\frac{{(-1)}^{ki} {\left(\gamma \right)}^i m^i (ki)!}{i!k^{ki}}
 x^{n-ki},
$$
we educe  a differential operator  formal power series
representation of the generalized hypergeometric Appell polynomials of the form of (\ref{scoef}).

\end{proof}

\textit{Remark}. Comparing  the power series (\ref{tree}) and operational formula(\ref{scoef}) of the generalized hypergeometric Appell polynomials to the corresponding  ones of the Gould-Hopper polynomials
$$A(t)=e^{ht^m}, \quad 
g_n^m (x,h)=\left( e^{hD^m} \right) x^n,
$$
it is easy to see that the latter  have  more compact forms.

\textbf{Symmetry.}
Substituting the negative value of argument into the formula (\ref{sumcoef})
\begin{equation*}
A_n^{(k)}(m,-x)=
\sum_{i=0}^{\left[ \slfrac{n}{k} \right]}
{(-1)}^{n-ki}
\frac{n!{(-1)}^{ki}{\left(\gamma \right)}^i m^i}{i!k^{ki}(n-ki)!}{x}^{n-ki},
\end{equation*}
we conclude that, in the case of even $k,$ the generalized hypergeometric Appell polynomials are the even ones themselves while $n$ is an even number, 
and they  are the odd ones themselves while $n$ is an odd number:
$$ A_{2n}^{(2k)}(m,-x)= A_{2n}^{(2k)}(m,x),
\quad  A_{2n+1}^{(2k)}(m,-x)= -A_{2n+1}^{(2k)}(m,x).$$

 Otherwise, for any odd $k$ in the case of odd $n,$ the summands standing on the even places change their signs into the opposite ones, and the same do the summands standing on the odd places in the case of even $n.$

\subsection{Addition and Multiplication Formulas and Other Properties}
Here we shall prove the following result.
\begin{theorem}
The following formulas hold for the  generalized hypergeometric Appell polynomials
 
$(i)$ \textit{addition formula}
$$
A_n^{(k)}(m,x+y)=
\sum_{i=0}^{n}
{n \choose i} y^{n-i}
A_i^{(k)}(m,x)=\sum_{i=0}^{n}
{n \choose i} x^{n-i}
A_i^{(k)}(m,y),
$$

$(ii)$ \textit{multiplication formula}
$$
A_n^{(k)}(m,Mx)=
\sum_{i=0}^{n}
{n \choose i}{(M-1)}^{n-i} x^{n-i}
A_i^{(k)}(m,x),
$$

$(iii)$ \textit{indexes interchange formula}
$$\sum_{i=0}^{n}
{n \choose i}
A_i^{({k_1})}(m,x)A_{n-i}^{({k_2})}(m,y)=\sum_{i=0}^{n}
{n \choose i}
A_i^{({k_2})}(m,x)A_{n-i}^{({k_1})}(m,y)
$$

$(iv)$ \textit{convolution type identity}
\begin{gather*}
\sum_{i=0}^{n} {(-1)}^i
{n \choose i}
A_i^{({k})}(m,x)A_{n-i}^{({k})}(m,x)\\
=\frac{{(-1)}^n m^{\slfrac{n}{k}}n!}{k^n}
\sum_{i=0}^{\left[ \slfrac{n}{k}\right]}
\frac{a_1^{(i)}  \ldots a_p^{(i)}}{i! b_1^{(i)}  \ldots b_q^{(i)}}
\frac{a_1^{\left(\slfrac{n}{k}-i \right)}  \ldots a_p^{\left(\slfrac{n}{k}-i \right)} }{\left(\slfrac{n}{k}-i \right)! b_1^{\left(\slfrac{n}{k}-i \right)}   \ldots b_q^{\left(\slfrac{n}{k}-i \right)} }.
\end{gather*}
\end{theorem}

\begin{proof}
The addition and the multiplication formulas hold for all Appell-type polynomial families (\cite{Haze}), consequently, they hold for the  generalized hypergeometric Appell polynomials as well. The indexes interchange formulas could be obtained applying methods  proposed in \cite{Cheih} and the convolution type identity is obtained by the simple direct calculations at $x=0$.
\end{proof}

It is worth stressing, that the polynomials $A_{n}^{(k)}(m,Mx)$ loose the property of being of Appell-type.
Moreover, the generalized hypergeometric  polynomials over the polynomials could be defined in the same manner as the generalized hypergeometric Appell polynomials:
\begin{equation*}
A_n^{(k)}(m,f(x))=
{\left(f(x)\right)}^n
{}_{p+k} F_q
\begin{bmatrix}
\begin{matrix}
{a_1}, {a_2}, {\ldots}, {a_p}, \Delta(k,-n)\:\\
{b_1}, {b_2}, {\ldots}, {b_q}
\end{matrix} \:
 \vrule \: \displaystyle  \frac{m}{{\left(f(x)\right)}^k}
\end{bmatrix},
\end{equation*}
where
$$f(x)=a_0x^p+a_1x^{p-1}+\cdots +a_p, \quad a_0\neq 0,$$
which submit the following differentiation rule

\begin{equation*}
\frac{d}{dx} A_n^{(k)}(m,f(x))=
n { f^{\prime}(x)} A_{n-1}^{(k)}(m,f(x)).
\end{equation*}

In particular, in the case when
$p=a_0=1,$ we obtain the Appell differentiation.

\end{document}